%% file: Hm1SpaceDecomposition.tex
\documentclass[11pt]{amsart}

\usepackage[utf8]{inputenc}
\usepackage[T1]{fontenc}

\usepackage{amsmath,amssymb,amsfonts,textcomp,amsthm,xifthen,graphicx,color,pgfplots}
\usepackage{ifthen}

\usepackage{enumerate}
\usepackage{enumitem}

\usepackage{fullpage}

\usepackage[utf8]{inputenc}

\usepackage[pdftex,
            pdfauthor={Thomas F\"uhrer and Norbert Heuer},
            pdftitle={Preconditioners for pseudodifferential operators of order minus two},
            ]{hyperref}

\usepackage{pgfplots}
\usepgfplotslibrary{external}
\input{header}

\begin{document}

\title{Optimal quasi-diagonal preconditioners for pseudodifferential operators of order minus two}
\date{\today}
\author{Thomas F\"{u}hrer}
\address{Facultad de Matem\'{a}ticas, Pontificia Universidad Cat\'{o}lica de Chile, Santiago, Chile}
\email{tofuhrer@mat.uc.cl \emph{(corresponding author)}, nheuer@mat.uc.cl}

\author{Norbert Heuer}
%

%
\thanks{{\bf Acknowledgment.} 
This work was supported by CONICYT through FONDECYT projects 11170050 and 1150056}

\begin{abstract}
  We present quasi-diagonal preconditioners for piecewise polynomial discretizations of
  pseudodifferential operators of order minus two in any space dimension.
  Here, quasi-diagonal means diagonal up to a sparse transformation.
  Considering shape regular simplicial meshes and arbitrary fixed polynomial degrees,
  we prove, for dimensions larger than one, that our preconditioners are asymptotically optimal.

  Numerical experiments in two, three and four dimensions confirm our results.
  For each dimension, we report on condition numbers for piecewise constant
  and piecewise linear polynomials.
\end{abstract}
\maketitle

\section{Introduction}
Let $A:\;H^{-1}(\Omega)\to H_0^1(\Omega)$ be a linear continuous and coercive operator, and
$f\in H_0^1(\Omega)$. Here, $\Omega\subset\R^n$ ($n\geq 2$) is a bounded connected polyhedral domain
with Lipschitz boundary, $H_0^1(\Omega)$ is the standard Sobolev space of $H^1(\Omega)$
functions with zero trace, and $H^{-1}(\Omega):=(H_0^1(\Omega))^*$ is its dual.
Considering the problem of finding $\phi\in H^{-1}(\Omega)$ such that
\begin{align} \label{prob}
  A\phi = f,
\end{align}
the finite element method (FEM) is a standard approach to approximate $\phi$. It consists in solving
the variational form of \cref{prob} in piecewise polynomial subspaces of $H^{-1}(\Omega)$.
The pseudodifferential operator $A$ is of order minus two and resulting linear systems are usually
ill conditioned.
For instance, using quasi-uniform meshes with elements of diameter $h$ and bounded polynomial
degrees, the FEM generates system matrices with spectral condition number growing like $O(h^{-2})$,
except specific basis functions are used, cf. Hsiao and Wendland
\cite[Corollary~2.1]{HsiaoW_81_ANL}, \cite[Remark~4]{HsiaoW_77_FEM}.
For a detailed analysis in the case of boundary integral operators, in particular considering locally
refined meshes, we refer to Ainsworth et al. \cite{amt99}.
Considering small mesh sizes $h$, there is an obvious need to use preconditioned
iterative solvers. In this paper, we show that very simple preconditioners yield uniformly
bounded condition numbers, for shape regular simplicial meshes in any space dimension
and for arbitrary, but fixed, polynomial degrees. Our preconditioners are diagonal up to
sparse transformations; we call them \emph{quasi-diagonal}.
We provide definitions and proofs for dimension $n\ge 2$. The one-dimensional case is ignored,
but can be deduced by straightforward simplifications.

Our interest in preconditioners for discretized operators of order minus two arose from our recent and
ongoing research. In~\cite{DPGplate} we proposed an ultraweak formulation of the
Kirchhoff--Love plate bending model where we consider two variables, the vertical deflection of the plate
and the bending moments. Both unknowns are taken in $L^2(\Omega)$, the standard Lebesgue space of
square integrable functions. In applications, also the shear force is a relevant quantity.
Being the divergence of the bending moments, it is generally not $L^2$-regular. Its natural
space is $H^{-1}(\Omega)^n$. Aiming at an approximation of the shear force, an efficient implementation
will require to study preconditioners in $H^{-1}(\Omega)$.

A second motivation is the approximation of obstacle problems by least-squares finite elements.
In~\cite{LSQobstacle}, a first-order reformulation with Lagrangian multiplier $\lambda$ was analyzed. 
The functional to be minimized there, includes a residual term measured in the $L^2(\Omega)$ norm.
However, measuring the residual in the weaker (discrete) $H^{-1}(\Omega)$ norm (as in~\cite{BLP97,BP96}) ensures 
optimal convergence orders for less regular solutions. This would lead to a different functional
\begin{align*}
   J_{-1}(u,\ssigma,\lambda) \simeq \norm{\nabla u}{}^2 + \norm{\ssigma}{}^2 + \norm{\lambda}{-1}^2,
   \quad (u,\ssigma,\lambda)\in H_0^1(\Omega)\times L^2(\Omega)^n\times H^{-1}(\Omega).
\end{align*}
An efficient implementation of a least-squares scheme for such a functional requires optimal preconditioners
for the $H^{-1}(\Omega)$ variable.

We note that the construction of preconditioners in Sobolev spaces of non-integer negative order is more
complicated than in $H^{-1}(\Omega)$. A standard case are weakly singular boundary integral operators
that are of order minus one. They are well posed as linear operators acting on $H^{-1/2}(\Gamma)$,
the dual space of the trace of $H^1(\Omega)$ when $\Gamma$ is the boundary of a sufficiently smooth
domain $\Omega$. Preconditioners consider a so-called coarse grid space and are of
two-level (or additive Schwarz) or multilevel type, see
\cite{TranS_96_ASM,HeuerST_98_MAS} for two-dimensional problems, and
\cite{Oswald_98_MNH,MundSW_98_TLM,Heuer_01_ApS,abemSolve} for problems in higher space dimensions.
Multigrid methods for two-dimensional problems have been analyzed in \cite{vonPetersdorffS_92_MSP},
and of algebraic construction in \cite{LangerPR_03_EPB}.
One has to note that in two-dimensions, where boundary integral operators live on curves,
the construction of preconditioners for weakly singular operators is equivalent to the one for
hypersingular operators, which are of opposite order, one.
Preconditioners constructed by using operators of opposite order have been proposed in
\cite{SteinbachW_98_CSE}. More recently, Stevenson and van Veneti\"e presented an abstract
theory for general negative orders and space dimensions based on the operator preconditioning framework, see~\cite{StevensonVanVenetiePrecond}.

In contrast to the aforementioned works, we consider simple (one-level) decompositions
of piecewise polynomials spaces of arbitrary (but fixed) order.
A key point is the decomposition of piecewise constants into the divergence of Raviart--Thomas basis functions.
We show that this decomposition is stable (in the sense of the additive Schwarz framework) in $H^{-1}(\Omega)$. 
Since our decomposition includes only one-dimensional spaces and since the support of the Raviart--Thomas basis
functions is local, it follows directly from the additive Schwarz theory that the corresponding preconditioner is
quasi-diagonal in the following sense:
For piecewise constants our proposed preconditioner has the form
\begin{align*}
  \PRECmat^{-1} = \Imat \Dmat \Imat^t,
\end{align*}
where $\Dmat$ is a diagonal matrix and $\Imat$ is a sparse matrix.
Let $\Amat$ denote the Galerkin matrix of the operator $A$ discretized with piecewise constants. From our analysis it
follows that the condition number is uniformly bounded, i.e.,
\begin{align*}
  \kappa(\PRECmat^{-1}\Amat) \leq C_\mathrm{cond} < \infty,
\end{align*}
where the constant depends in general on $\Omega$, the dimension $n\in\N$, the polynomial degree $p\in\N_0$ and
the shape regularity of the underlying mesh.
We note that an optimal (local) multilevel diagonal preconditioner for the weakly singular integral operator
has been proposed in~\cite{abemSolve}.
It is based on the decomposition of piecewise constants into the surface divergence of Raviart--Thomas functions.

The remainder of this paper is as follows. Our discrete settings with definition of spaces,
subspace decompositions as well as the main results are formulated in the next section.
Specifically, Sobolev spaces and norms are recalled in \cref{sec:main:norm},
discrete spaces are defined in \cref{sec:main:discrete} along with collecting some norm relations,
and basic additive Schwarz settings are given in \cref{sec:abstract}. Subsequently, our
four principal results on bounded condition numbers are formulated in four steps,
namely for spaces of piecewise polynomials of degrees $p=0$ in $H^{-1}(\Omega)$ and $\widetilde H^{-1}(\Omega)$
(the dual of $H^1(\Omega)$), and for higher degrees in $H^{-1}(\Omega)$ and $\widetilde H^{-1}(\Omega)$,
respectively, in \cref{sec:decomp:Hm1,sec:decomp:tildeHm1,sec:decomp:Hm1:Pp,sec:decomp:tildeHm1:Pp}.
Proofs of the main results are given in \cref{sec:proof:Hm1,sec:proof:tildeHm1,sec:proof:Pp}.
For our numerical results we need explicit matrix representations of the preconditioners.
They are given in \cref{sec:matrix}. Finally, in \cref{sec:num}, we report on numerical experiments
in two, three and four space dimensions for piecewise constant and piecewise linear polynomials.
We study uniform refined meshes for all cases and locally refined meshes for the two-dimensional case.
In the appendix we give a proof of a technical result needed in the analysis.

Throughout, the notation $a\lesssim b$ means that there exists a constant $c>0$
such that $a\le cb$. The constant is independent of the underlying mesh under the assumption
of (uniform) shape regularity, but may depend on the polynomial degrees, the space dimension
and the domain $\Omega$.
The relation $a\simeq b$ means that $a\lesssim b$ and $b\lesssim a$.

\section{Main results}\label{sec:main}

\subsection{Sobolev spaces \& (semi-)norms} \label{sec:main:norm}
The boundary of the Lipschitz polyhedron $\Omega\subset \R^n$ ($n\geq 2$) is denoted by $\Gamma$,
and $\normal$ is the unit normal vector on $\Gamma$ pointing outside of $\Omega$.
For a non-empty open and connected subset $\omega\subseteq\Omega$, we denote the $L^2(\omega)$
scalar product and norm by $\ip\cdot\cdot_\omega$ and $\norm{\cdot}\omega$, respectively.
If $\omega=\Omega$ we skip the index, i.e., $\ip\cdot\cdot =
\ip\cdot\cdot_\Omega$ and $\norm\cdot{} =\norm{\cdot}\Omega$.
Furthermore, we also use the space $L_*^2(\omega) := \set{v\in L^2(\Omega)}{\ip{v}1_\omega=0}$.
For $m\in\N$, $H^m(\omega)$ denotes the standard Sobolev space of $m$ times weakly differentiable
functions, with norm $\norm{\cdot}{m,\omega}$.
Again, if $\omega=\Omega$, then $\norm{\cdot}{m} = \norm{\cdot}{m,\Omega}$.
We consider also fractional-order Sobolev spaces $H^{m+s}(\omega)$ ($m\in\N_0,s\in(0,1)$)
with (squared) Sobolev--Slobodetskij seminorm 
\begin{align*}
  \snorm{v}{m+s,\omega}^2 := \sum_{|\boldsymbol{\alpha}|=m} \int_\omega\int_\omega
  \frac{|D^{\boldsymbol{\alpha}}v(x) - D^{\boldsymbol\alpha}v(y)|}{|x-y|^{n+2s}} \,dy\,dx
\end{align*}
and norm $\norm{\cdot}{m+s,\omega}^2 := \norm{\cdot}{m,\omega}^2 + \snorm{\cdot}{m+s,\omega}^2$.
For $m=0$ we identify $H^0(\omega) = L^2(\omega)$, and if $\omega=\Omega$, we skip the index.
For $s>0$, the dual spaces are $\widetilde H^{-s}(\omega) = (H^s(\omega))^*$ and
the duality pairing is given by the extended $L^2(\omega)$ scalar product.
The dual spaces are equipped with the norms
\begin{align*}
  \norm{\phi}{-s,\sim,\omega} := \sup_{0\neq v\in H^s(\Omega)} \frac{\ip{\phi}{v}}{\norm{v}{s,\omega}} \quad(s>0).
\end{align*}
For $s>0$, the spaces $H_0^s(\omega)$ are defined as the completion of $C_0^\infty(\Omega)$
with respect to the norm $\norm{\cdot}{s,\omega}$.
The dual spaces are denoted by $H^{-s}(\omega) := (H_0^s(\omega))^*$ with dual norms $\norm{\cdot}{-s,\omega}$.
In the special case $\omega=\Omega$ and $s=1$ we use the dual norm
\begin{align*}
  \norm{\phi}{-1} := \sup_{0\neq v\in H_0^1(\Omega)} \frac{\ip{\phi}v}{\norm{\nabla v}{}}.
\end{align*}
The scalar products in the dual spaces, inducing the norms $\norm{\cdot}{-1,\sim}$ and $\norm{\cdot}{-1}$,
are denoted by $\ip{\cdot}\cdot_{-1,\sim}$ and $\ip{\cdot}\cdot_{-1}$, respectively.

\subsection{Mesh \& discrete spaces} \label{sec:main:discrete}
Let $\TT$ denote a regular mesh of open $n$-simplices that cover $\Omega$, i.e.,
\begin{align*}
  \overline\Omega = \bigcup_{T\in\TT} \overline T.
\end{align*}
An open $n$-simplex $T\in\TT$ is the interior of the convex hull of $n+1$ different vertices
$\zz_{T,j}\in\R^n$ ($j=1,\dots,n+1$) that do not lie on the same hypersurface.
We say that $\TT$ is shape regular if there exists a positive constant $\gamma$ such that
\begin{align*}
  \max_{T\in\TT} \frac{\diam(T)^n}{|T|} \leq \gamma < \infty.
\end{align*}
Here, $|T|$ denotes the measure (volume) of $T\in\TT$. 
By $\EE(T)$ we denote the set of all $n+1$ boundary simplices of $T$
(relatively open simplices made up of $n$ vertices of $T$). The elements $E\in\EE(T)$ are called facets.
The collection of all facets is denoted by $\EE = \EE(\TT) = \bigcup_{T\in\TT}\bigcup_{E\in\EE(T)} \{E\}$.
With $\EE^\Gamma$ we denote all facets in $\EE$ that are subsets of $\Gamma$, while
$\EE^\Omega := \EE\setminus \EE^\Gamma$ is the set of interior facets. For $T\in\TT$ we define the patch
\begin{align*}
  \patch(T) := \set{T'\in\TT}{\overline T' \cap \overline T \neq \emptyset}.
\end{align*}
The mesh-width function $h = h_\TT$ and the local mesh-width $h_T$ are given by
\begin{align*}
  h|_T := h_T := \diam(T) \quad\text{for all } T\in\TT.
\end{align*}
We note that $h_{T'}\simeq h_T$ for all $T'\in \patch(T)$, 
and $h_E := \diam(E) \simeq h_T$ for all $E\in \EE(T)$, where the equivalence constants only depend on
the shape regularity of $\TT$.
For consistency we identify the patch $\patch(T)$ with the domain 
$\mathrm{int}(\bigcup_{T'\in\patch(T)} \overline T')$ where needed, e.g., to refer to inner products on
patches as in $\ip\cdot\cdot_{\patch(T)}$.

Let $\PP^p(\TT)$ denote the space of $\TT$-elementwise polynomials of degree less than or equal to $p\in\N_0$.
Correspondingly, $\RT^0(\TT)$ is the lowest-order Raviart--Thomas space (basis functions are defined below).
The space $\PP^0(\TT)$ is equipped with the standard basis of characteristic functions
$\set{\chi_T}{T\in\TT}$, where
\begin{align*}
  \chi_T|_{T'} = \begin{cases}
    1 & \text{if } T'=T, \\
    0 & \text{else}.
  \end{cases}
\end{align*}
For $p\geq 1$, we denote $d(n,p):=\dim(\PP^p(T))$ (which is independent of $T\in\TT$), and let
\begin{align*}
  \set{\chi_T,\chi_{T,1},\chi_{T,2},\dots,\chi_{T,d(n,p)-1}}{T\in\TT}
\end{align*}
be a basis of $\PP^p(\TT)$ with $\supp\{\chi_{T,j}\}=\overline T$ and normalized by $\ip{\chi_{T,j}}1 = 0$
($T\in\TT$, $j\in \JJ := \{1,\cdots,d(n,p)-1\}$). Recall that
\begin{align*}
  d(n,p) = {p+n\choose p} = \frac{\prod_{j=1}^n (p+j)}{n!}.
\end{align*}
Throughout our analysis, we will make use of the inverse inequalities
\begin{align}\label{eq:invineq}
  \norm{h\phi}{} \lesssim \norm{\phi}{-1} \quad\text{and}\quad
  \norm{h\phi}{} \lesssim \norm{\phi}{-1,\sim} \quad\text{for all }
  \phi \in \PP^p(\TT),
\end{align}
see, e.g.,~\cite[Theorem~3.6]{ghs05} for a general case with $n=3$.
We stress the fact that these relations follow by simple scaling arguments.
The involved constants only depend on the shape regularity of $\TT$, $n\in\N$, and $p\in\N_0$.

For each $E\in \EE^\Omega$ there exist exactly two elements $T^\pm$ with $\overline T^+\cap \overline T^- = \overline
E$. 
Furthermore, let $\pp_E^\pm\in \overline T^\pm$ denote the vertex of $T^\pm$ opposite to $E$ and let $\normal_E$ denote the
unit normal on $E$ pointing from $T^+$ to $T^-$.
If $E\in \EE^\Gamma$ then there is only one element $T\in\TT$ such that $E$ is a facet of $T$. 
In such a case we write $T^+= T$, $T^-=\emptyset$ and note that $\normal_E$ points from $\Omega$ to the exterior,
i.e., $\normal_E$ coincides with the normal vector $\normal$ on $\Gamma$.
Let $|E|$ denote the (relative) measure of $E\in\EE$.
For $n\geq 2$ we define the Raviart--Thomas basis function
\begin{align*}
  \ppsi_E := \begin{cases}
    \pm \frac{|E|}{n|T^\pm|} \left(\xx - \pp_E^\pm\right) & \text{if } \xx \in T^+\cup T^-, \\
    0 & \text{else},
  \end{cases}
\end{align*}
and note that
\begin{align*}
  \div\ppsi_E = \begin{cases}
    \pm \frac{|E|}{|T^\pm|} & \text{if } \xx\in T^+\cup T^-, \\
    0 & \text{else}.
  \end{cases}
\end{align*}
The following scaling result on the Raviart--Thomas functions will be used several times.
\begin{lemma}\label{lem:scaling}
  It holds that
  \begin{alignat*}{2}
    \norm{\div\ppsi_E}{-1} &\leq \norm{\ppsi_E}{} \simeq h_E \norm{\div\ppsi_E}{} 
    \lesssim \norm{\div\ppsi_E}{-1} &\quad&\text{for all } E\in\EE, \\
    \norm{\div\ppsi_E}{-1,\sim} &\leq \norm{\ppsi_E}{} \simeq h_E \norm{\div\ppsi_E}{} 
    \lesssim \norm{\div\ppsi_E}{-1,\sim} &\quad&\text{for all } E\in\EE^\Omega.
  \end{alignat*}
  The involved constants only depend on the shape regularity of $\TT$ and $n$.
\end{lemma}
\begin{proof}
  By definition of the dual norm and integration by parts we see that
  \begin{align*}
    \norm{\div\ppsi_E}{-1} = \sup_{0\neq v\in H_0^1(\Omega)} \frac{\ip{\div\ppsi_E}{v}}{\norm{\nabla v}{}}
    = \sup_{0\neq v\in H_0^1(\Omega)} \frac{-\ip{\ppsi_E}{\nabla v}}{\norm{\nabla v}{}}
    \leq \norm{\ppsi_E}{}
  \end{align*}
  for all $E\in \EE$. If $E\in \EE^\Omega$ the same argument shows that
  \begin{align*}
    \norm{\div\ppsi_E}{-1,\sim} = \sup_{0\neq v\in H^1(\Omega)} \frac{\ip{\div\ppsi_E}{v}}{\norm{v}{1}}
    = \sup_{0\neq v\in H^1(\Omega)} \frac{-\ip{\ppsi_E}{\nabla v}}{\sqrt{\norm{\nabla v}{}^2+\norm{v}{}^2}}
    \leq \norm{\ppsi_E}{}.
  \end{align*}
  Furthermore, the equivalence $\norm{\ppsi_E}{} \simeq h_E \norm{\div\ppsi_E}{}$ follows by a simple scaling argument.
  Finally, note that $h_E\norm{\div\ppsi_E}{} \simeq \norm{h\div\ppsi_E}{}$. Together with the inverse
  inequality~\cref{eq:invineq} we conclude the proof.
\end{proof}

The following equivalence result in $L^2$ follows from the equivalence of norms in finite-dimensional
spaces and linear independence.
\begin{proposition}
  Let $T\in\TT$.
  If $\ppsi = \sum_{E\in\EE(T)} \alpha_E\ppsi_E$, then
  \begin{align*}
    \norm{\ppsi}{T}^2 \simeq \sum_{E\in\EE(T)} \norm{\alpha_E\ppsi_E}{T}^2,
  \end{align*}
  where the involved constants only depend on the shape regularity of $\TT$ and $n$.
  \qed
\end{proposition}

We recall that for $s>1/2$ the Raviart--Thomas projector $\Pi^\div : H^s(\Omega)^n \to \RT^0(\TT)$
is given by
\begin{align*}
  \Pi^\div\ssigma := \sum_{E\in\EE} \alpha_E \ppsi_E \quad\text{with}\quad \alpha_E := \frac{1}{|E|} 
  \int_E \ssigma\cdot\normal_E \,ds_E.
\end{align*}
It holds the following local bound on the $L^2$ norm,
\begin{align}\label{eq:RT:localbound}
  \norm{\Pi^\div\ssigma}{T} \lesssim \norm{\ssigma}{T} + h_T^s\snorm{\ssigma}{s,T}.
\end{align}
Here, the involved constants only depend on the shape regularity of $\TT$, $n$ and $s\in(1/2,1]$.
Also recall the commutativity property $\Pi^0\div\ssigma = \div\Pi^\div\ssigma$ for sufficiently smooth $\ssigma$,
where $\Pi^0 : L^2(\Omega) \to \PP^0(\TT)$ is the $L^2(\Omega)$ projection.
These results can be found, e.g., in \cite[Chapter~2]{BoffiBrezziFortin} and easily extend to arbitrary space
dimensions.

\subsection{Additive Schwarz framework}\label{sec:abstract}
Let $\XX$ denote a finite-dimensional subspace of a Hilbert space $\HHH$ with norm $\norm\cdot{\HHH}$.
For an index set $\II$ let $\XX_i\subseteq \XX$ ($i\in\II$) denote subspaces of $\XX$ such that
we have the splitting
\begin{align}\label{eq:splitting:abstract}
  \XX = \sum_{i\in\II} \XX_i.
\end{align}
To this decomposition we associate the \emph{additive Schwarz norm} $\enorm\cdot_{\XX}$ given by
\begin{align}\label{eq:ASnorm:abstract}
  \enorm{x}_{\XX}^2 := \inf\set{\sum_{i\in\II} \norm{x_i}{\HHH}^2}{x_i\in \XX_i \text{ such that } x = \sum_{i\in\II} x_i}.
\end{align}
Central to the additive Schwarz framework is to establish a norm equivalence of the form
\begin{align}\label{eq:equiv:abstract}
  C^{-1} \enorm{x}_{\XX} \leq \norm{x}{\HHH} \leq C \enorm{x}_{\XX} \quad\text{for all }x\in \XX.
\end{align}
Having such an equivalence implies that one can define related additive Schwarz 
preconditioners for elliptic problems in $\HHH$.
The condition numbers of the resulting preconditioned system matrices only depend on $C>0$. 
This is well-known knowledge and we refer the interested reader to~\cite{oswald94,ToselliWidlund}.

\subsection{Subspace decomposition of $\PP^0(\TT)$ in $H^{-1}(\Omega)$}\label{sec:decomp:Hm1}
Set $\XX_E := \linhull\{\div\ppsi_E\}\subseteq \PP^0(\TT) =: \XX$.
We consider the splitting
\begin{align}\label{eq:splitting:Hm1}
  \XX = \sum_{E\in\EE} \XX_E
\end{align}
and the associated norm~\cref{eq:ASnorm:abstract}.
Noting that $\div(\RT^0(\TT)) = \PP^0(\TT)$ this shows that the sum on the right-hand side is indeed a decomposition of
$\PP^0(\TT)$.

Our first main result reads as follows.
\begin{theorem}\label{thm:Hm1}
  There exists $C>0$ which only depends on $\Omega$, $n$, and the shape regularity of $\TT$ such that
  \begin{align}\label{eq:equiv:Hm1}
    C^{-1} \enorm{\phi}_{\XX} \leq \norm{\phi}{-1} \leq C \enorm{\phi}_{\XX} \quad\text{for all } 
    \phi \in \PP^0(\TT).
  \end{align}
\end{theorem}

\subsection{Subspace decomposition of $\PP^0(\TT)$ in $\widetilde H^{-1}(\Omega)$}\label{sec:decomp:tildeHm1}
Set $\XX_0 := \linhull\{1\}\subseteq \PP^0(\TT) =: \YY$.
We consider the splitting
\begin{align}\label{eq:splitting:tildeHm1}
  \YY = \XX_0 + \sum_{E\in\EE^\Omega} \XX_E
\end{align}
and the associated norm $\enorm\cdot_{\YY}$.
Again, the right-hand side of~\cref{eq:splitting:tildeHm1} defines a decomposition of
$\PP^0(\TT)$. A proof is omitted since it is implicitly contained in the proof of our second main result.
(In~\cref{sec:proof:tildeHm1}, a decomposition is constructed for any $\phi\in \PP^0(\TT)$.)
\begin{theorem}\label{thm:tildeHm1}
  There exists $C>0$ which only depends on $\Omega$, $n$, and the shape regularity of $\TT$ such that
  \begin{align}\label{eq:equiv:tildeHm1}
    C^{-1} \enorm{\phi}_{\YY} \leq \norm{\phi}{-1,\sim} \leq C \enorm{\phi}_{\YY} \quad\text{for all } 
    \phi \in \PP^0(\TT).
  \end{align}
\end{theorem}

\subsection{Subspace decomposition of $\PP^p(\TT)$ in $H^{-1}(\Omega)$}\label{sec:decomp:Hm1:Pp}
Let $\XX_p := \PP^p(\TT)$, $\XX_{T,j} := \linhull\{\chi_{T,j}\}$. We consider the decomposition
\begin{align*}
  \XX_p = \sum_{E\in\EE} \XX_E + \sum_{T\in\TT} \sum_{j\in \JJ} \XX_{T,j} \quad\text{with associated norm }
  \enorm{\cdot}_{\XX_p}.
\end{align*}

\begin{theorem}\label{thm:Hm1:Pp}
  There exists $C>0$ which only depends on $\Omega$, $n$, $p$, and the shape regularity of $\TT$ such that
  \begin{align}\label{eq:equiv:Hm1:Pp}
    C^{-1} \enorm{\phi}_{\XX_p} \leq \norm{\phi}{-1} \leq C \enorm{\phi}_{\XX_p} \quad\text{for all } 
    \phi \in \PP^p(\TT).
  \end{align}
\end{theorem}

\subsection{Subspace decomposition of $\PP^p(\TT)$ in $\widetilde H^{-1}(\Omega)$}\label{sec:decomp:tildeHm1:Pp}
There holds a similar decomposition:
\begin{align*}
  \YY_p = \XX_0  +  \sum_{E\in\EE^\Omega} \XX_E + \sum_{T\in\TT} \sum_{j\in \JJ} \XX_{T,j}
  \quad\text{with associated norm } \enorm{\cdot}_{\YY_p}.
\end{align*}

\begin{theorem}\label{thm:tildeHm1:Pp}
  There exists $C>0$ which only depends on $\Omega$, $n$, $p$, and the shape regularity of $\TT$ such that
  \begin{align}\label{eq:equiv:tildeHm1:Pp}
    C^{-1} \enorm{\phi}_{\YY_p} \leq \norm{\phi}{-1,\sim} \leq C \enorm{\phi}_{\YY_p} \quad\text{for all } 
    \phi \in \PP^p(\TT).
  \end{align}
\end{theorem}

\section{Proof of~\cref{thm:Hm1}}\label{sec:proof:Hm1}
Before we come to the proof let us collect some auxiliary results.
Let $f\in L^2(\Omega)$ be given and let $u\in H_0^1(\Omega)$ denote the unique weak solution of the Poisson equation
\begin{align}\label{eq:dirichlet}
  -\Delta u = f \quad\text{in }\Omega \quad\text{and}\quad u|_\Gamma = 0.
\end{align}
From elliptic regularity theory \cite{grisvard,dauge88} we know that there exists a regularity shift
$s=s(\Omega)\in(1/2,1]$ such that
\begin{align}\label{eq:regularity}
  \norm{u}{1+r} \lesssim \norm{f}{-1+r} \quad\text{for all } r\in[0,s].
\end{align}
We need the following local regularity result. 
Its proof follows along the argumentation given
in~\cite[Theorem~3.3]{AinsworthGuzmanSayas} with only minor modifications.
(The difference is that in~\cite{AinsworthGuzmanSayas} the authors consider the Poisson equation with $f=0$ and
non-trivial Neumann datum for $n=3$, whereas here we consider non-trivial $f$ and homogeneous Dirichlet datum for
$n\geq 2$. For completeness we give a proof in~\cref{app:localreg}.)
\begin{lemma}\label{lem:localreg}
  Let $f\in L^2(\Omega)$ and let $u\in H_0^1(\Omega)$ denote the solution of~\cref{eq:dirichlet}.
  It holds that 
  \begin{align*}
    h_T^s \snorm{\nabla u}{s,T} \lesssim \norm{\nabla u}{\patch(T)} + h_T \norm{f}{\patch(T)}
    \quad\text{for all } T\in\TT.
  \end{align*}
  Here, the involved constant only depends on $\Omega$ and the shape regularity of $\TT$.
  \qed
\end{lemma}

\subsection{Proof of lower bound in~\cref{eq:equiv:Hm1}}
Let $\phi \in \PP^0(\TT)$ be given. Define $u\in H_0^1(\Omega)$ as the solution of~\cref{eq:dirichlet} with right-hand
side $f=-\phi$.
By~\eqref{eq:regularity} it holds that $u\in H^{1+s}(\Omega)$.
In particular $\ssigma:= \nabla u \in H^s(\Omega)^n$ and $\Pi^{\div}\ssigma$ is well defined.
Recall that
\begin{align*}
  \Pi^{\div}\ssigma = \sum_{E\in\EE} \alpha_E \ppsi_E, \text{ where } 
  \alpha_E = \frac1{|E|}\int_E \ssigma\cdot\normal_E \,ds_E.
\end{align*}
We set $\phi_E := \alpha_E \div\ppsi_E \in \XX_E$. Observe that (using the commutativity property of $\Pi^{\div}$)
\begin{align*}
  \sum_{E\in\EE} \phi_E = \sum_{E\in\EE} \div(\alpha_E\ppsi_E) = \div(\Pi^{\div}\ssigma) = \Pi^0\div\ssigma = \phi.
\end{align*}
Then,~\cref{lem:scaling} shows that
\begin{align*}
  \sum_{E\in\EE} \norm{\phi_E}{-1}^2 &= \sum_{E\in\EE} \norm{\alpha_E\div\ppsi_E}{-1}^2 \leq 
  \sum_{E\in\EE} \norm{\alpha_E\ppsi_E}{}^2 = \sum_{T\in\TT} \sum_{E\in\EE(T)} \norm{\alpha_E \ppsi_E}{T}^2 \\
  &\simeq \sum_{T\in\TT} \norm{\sum_{E\in\EE(T)}\alpha_E\ppsi_E}{T}^2 = \sum_{T\in\TT} \norm{\Pi^{\div}\ssigma}T^2.
\end{align*}
Using~\cref{eq:RT:localbound} together with~\cref{lem:localreg} we estimate the last term further by
\begin{align*}
  \sum_{T\in\TT} \norm{\Pi^{\div}\ssigma}T^2 &\lesssim \sum_{T\in\TT} \left(\norm{\ssigma}T^2 +
  h_T^{2s}\snorm{\ssigma}{s,T}^2 \right)
  \lesssim \norm{\nabla u}{}^2 + 
  \sum_{T\in\TT} h_T^2\norm{\phi}{T}^2.
\end{align*}
A standard estimate gives $\norm{\nabla u}{} \leq \norm{\phi}{-1}$ and with the inverse inequality~\cref{eq:invineq},
i.e.,
\begin{align*}
  \sum_{T\in\TT} h_T^2\norm{\phi}{T}^2 \lesssim \norm{\phi}{-1}^2,
\end{align*}
we conclude the proof of the lower bound in~\cref{eq:equiv:Hm1}. \qed

\subsection{Proof of upper bound in~\cref{eq:equiv:Hm1}}
Let $\phi\in \PP^0(\TT)$ and let $\phi_E\in\XX_E$ be arbitrary such that $\phi = \sum_{E\in\EE} \phi_E$.
Note that we can write $\phi_E = \alpha_E \div\ppsi_E$. This shows that
\begin{align*}
  \phi = \div(\sum_{E\in\EE}\alpha_E\ppsi_E).
\end{align*}
Using $\norm{\div(\cdot)}{-1}\leq \norm{\cdot}{}$ we get
\begin{align*}
  \norm{\phi}{-1}^2 &\leq \norm{\sum_{E\in\EE}\alpha_E\ppsi_E}{}^2 = \sum_{T\in\TT}
  \norm{\sum_{E\in\EE(T)}\alpha_E\ppsi_E}T^2 
  \lesssim \sum_{T\in\TT} \sum_{E\in\EE(T)} \norm{\alpha_E\ppsi_E}T^2 = \sum_{E\in\EE}\norm{\alpha_E\ppsi_E}{}^2.
\end{align*}
With the estimate
\begin{align*}
  \norm{\ppsi_E}{}\simeq h_T\norm{\div\ppsi_E}{} \lesssim \norm{\div\ppsi_E}{-1},
\end{align*}
see~\cref{lem:scaling}, we conclude the proof of the upper bound in~\cref{eq:equiv:Hm1}. 
\qed

\section{Proof of~\cref{thm:tildeHm1}}\label{sec:proof:tildeHm1}
The proof follows similar ideas as in~\cref{sec:proof:Hm1}.
Instead of the Dirichlet problem~\cref{eq:dirichlet}, we consider the following Neumann problem.
Let $f\in L_*^2(\Omega)$ be given and let $u\in H^1(\Omega)\cap L_*^2(\Omega)$ denote the unique weak solution of the
Poisson problem
\begin{align}\label{eq:neumann}
  -\Delta u = f \quad\text{in }\Omega \quad\text{and}\quad \nabla u\cdot\normal|_\Gamma = 0.
\end{align}
We note that the regularity~\cref{eq:regularity} also holds true for solutions of~\cref{eq:neumann}.

As in~\cref{sec:proof:Hm1} the proof of the next result follows along the argumentation given
in~\cite[Theorem~3.3]{AinsworthGuzmanSayas}.
\begin{lemma}\label{lem:localreg:neumann}
  Let $f\in L_*^2(\Omega)$ and let $u\in H^1(\Omega)\cap L_*^2(\Omega)$ denote the solution of~\cref{eq:neumann}.
  It holds that 
  \begin{align*}
    h_T^s \snorm{\nabla u}{s,T} \lesssim \norm{\nabla u}{\patch(T)} + h_T \norm{f}{\patch(T)}
    \quad\text{for all } T\in\TT.
  \end{align*}
  Here, the involved constant only depends on $\Omega$ and the shape regularity of $\TT$.
  \qed
\end{lemma}

Throughout the remainder of this section, let $\phi \in \PP^0(\TT)$ be given.
We consider the (unique) splitting $\phi = \phi_0 + \phi_*$ where $\phi_0 := \frac1{|\Omega|}\int_\Omega \phi \,dx$.
Clearly, $\int_\Omega \phi_* \,dx = 0$.
Observe that
\begin{align*}
  \norm{\phi_0}{-1,\sim} \lesssim \norm{\phi}{-1,\sim},
\end{align*}
hence,
\begin{align}\label{eq:splitting:const}
  \norm{\phi}{-1,\sim} \simeq \norm{\phi_0}{-1,\sim} + \norm{\phi_*}{-1,\sim},
\end{align}
where the involved constants only depend on $\Omega$.

\subsection{Proof of lower bound in~\cref{eq:equiv:tildeHm1}}
Define $u\in H^1(\Omega)\cap L_*^2(\Omega)$ 
as the solution of~\cref{eq:neumann} with right-hand
side $f=-\phi_*$.
By~\eqref{eq:regularity} we have that $u\in H^{1+s}(\Omega)$, where $s\in(1/2,1]$ denotes the regularity shift.
In particular, $\ssigma:= \nabla u \in H^s(\Omega)^n$ and $\Pi^{\div}\ssigma$ is well defined.
Recall that
\begin{align*}
  \Pi^{\div}\ssigma = \sum_{E\in\EE} \alpha_E \ppsi_E, \text{ where } 
  \alpha_E = \frac1{|E|}\int_E \ssigma\cdot\normal_E \,ds_E.
\end{align*}
Note that $\ssigma\cdot\normal = \nabla u\cdot\normal = 0$ on the boundary $\Gamma$ and therefore,
$\alpha_E = 0$ for all $E\in\EE^\Gamma$.
We set $\phi_E := \alpha_E \div\ppsi_E \in \XX_E$. Observe that (using the commutativity property of $\Pi^{\div}$)
\begin{align*}
  \sum_{E\in\EE^\Omega} \phi_E = \sum_{E\in\EE^\Omega} \div(\alpha_E\ppsi_E) = \sum_{E\in\EE} \div(\alpha_E\ppsi_E) 
  = \div(\Pi^{\div}\ssigma) = \Pi^0\div\ssigma = \phi_*.
\end{align*}
Then, keeping in mind that $\alpha_E=0$ for $E\in\EE^\Gamma$,~\cref{lem:scaling} yields
\begin{align*}
  \sum_{E\in\EE^\Omega} \norm{\phi_E}{-1,\sim}^2 &= \sum_{E\in\EE} \norm{\alpha_E\div\ppsi_E}{-1,\sim}^2 \leq 
  \sum_{E\in\EE} \norm{\alpha_E\ppsi_E}{}^2 = \sum_{T\in\TT} \sum_{E\in\EE(T)} \norm{\alpha_E \ppsi_E}{T}^2 \\
  &\simeq \sum_{T\in\TT} \norm{\sum_{E\in\EE(T)}\alpha_E\ppsi_E}{T}^2 = \sum_{T\in\TT} \norm{\Pi^{\div}\ssigma}T^2.
\end{align*}
Using~\cref{eq:RT:localbound} together with~\cref{lem:localreg:neumann}
we can bound the last term as in~\cref{sec:proof:Hm1}.
Finally, with the equivalence~\cref{eq:splitting:const}
we conclude the proof of the lower bound in~\cref{eq:equiv:tildeHm1}.
\qed

\subsection{Proof of upper bound in~\cref{eq:equiv:tildeHm1}}
Let $\phi_E\in\XX_E$ be arbitrary such that $\phi_* = \sum_{E\in\EE^\Omega} \phi_E$.
Note that we can write $\phi_E = \alpha_E \div\ppsi_E$. This shows that
\begin{align*}
  \phi_* = \div(\sum_{E\in\EE^\Omega}\alpha_E\ppsi_E).
\end{align*}
Since $(\sum_{E\in\EE^\Omega}\alpha_E\ppsi_E)\cdot\normal = 0$ on $\Gamma$ it holds that
\begin{align*}
  \norm{\div\sum_{E\in\EE^\Omega}\alpha_E\ppsi_E}{-1,\sim}\leq \norm{\sum_{E\in\EE^\Omega}\alpha_E\ppsi_E}{}.
\end{align*}
Arguing as in~\cref{sec:proof:Hm1} together with the equivalence~\cref{eq:splitting:const} we finish the proof of the
upper bound in~\cref{eq:equiv:tildeHm1}. 
\qed

\section{Proof of~\cref{thm:Hm1:Pp,thm:tildeHm1:Pp}} \label{sec:proof:Pp}
We only consider the proof of~\cref{thm:Hm1:Pp}. \cref{thm:tildeHm1:Pp} can be shown analogously.

\begin{lemma}\label{lem:L2proj}
  The $L^2(\Omega)$ projection $\Pi^0$ restricted to $\PP^p(\TT)$ is bounded in $H^{-1}(\Omega)$.
  In particular, 
  \begin{align*}
    \norm{(1-\Pi^0)\phi}{-1}
    &\leq C_1 \norm{h(1-\Pi^0)\phi}{} \leq C_1\norm{h\phi}{} \leq C_1 C_2 \norm{\phi}{-1}, \\
    \norm{\Pi^0\phi}{-1} &\leq (1+C_1C_2)\norm{\phi}{-1}
  \end{align*}
  for all $\phi \in \PP^p(\TT)$, 
  where $C_1,C_2$ only depend on the shape regularity of $\TT$ and $C_2$ also depends on $p\in\N_0$ and $n$.
\end{lemma}
\begin{proof}
  With the local approximation property $\norm{(1-\Pi^0)v}{T}\lesssim h_T\norm{\nabla v}T$ of the $L^2$ projector and the estimate
  $\norm{(1-\Pi^0)\phi}T\leq \norm{\phi}T$ we have
  \begin{align*}
    \norm{(1-\Pi^0)\phi}{-1} &= \sup_{0\neq v\in H_0^1(\Omega)} \frac{\ip{(1-\Pi^0)\phi}v}{\norm{\nabla v}{}}
    = \sup_{0\neq v\in H_0^1(\Omega)} \frac{\ip{(1-\Pi^0)\phi}{(1-\Pi^0)v}}{\norm{\nabla v}{}}
    \\
    &\lesssim \norm{h(1-\Pi^0)\phi}{} \leq \norm{h\phi}{}.
  \end{align*}
  The inverse inequality~\cref{eq:invineq} shows the first estimate.
  Using the triangle inequality we conclude the second estimate.
\end{proof}

\subsection{Proof of lower bound in~\cref{eq:equiv:Hm1:Pp}}
Let $\phi\in \PP^p(\TT)$ be given. We consider the $L^2$ orthogonal splitting
\begin{align*}
  \phi := \phi_0 + \phi_1 = \Pi^0\phi + (1-\Pi^0)\phi.
\end{align*}
From~\cref{thm:Hm1} we already know that we can split $\phi_0 = \sum_{E\in\EE} \phi_E$ such that
\begin{align*}
  \sum_{E\in\EE} \norm{\phi_E}{-1}^2 \lesssim \norm{\phi_0}{-1}^2 \lesssim \norm{\phi}{-1}^2,
\end{align*}
where for the last estimate we have used~\cref{lem:L2proj}.
The function $\phi_1$ is uniquely decomposed as
\begin{align*}
  \phi_1 = \sum_{T\in\TT} \sum_{j\in\JJ} \phi_{T,j} \quad\text{with } \phi_{T,j}\in \XX_{T,j}.
\end{align*}
Recall that $\ip{\phi_{T,j}}1 = 0$, hence, $\phi_{T,j} = (1-\Pi^0)\phi_{T,j}$ and with~\cref{lem:L2proj} we infer that
\begin{align*}
  \norm{\phi_{T,j}}{-1} \lesssim h_T\norm{\phi_{T,j}}{T}.
\end{align*}
Together with the inverse inequality~\cref{eq:invineq} we infer that
\begin{align*}
  \sum_{T\in\TT} \sum_{j\in\JJ} \norm{\phi_{T,j}}{-1}^2 \lesssim \sum_{T\in\TT} h_T^2 \sum_{j\in\JJ}
  \norm{\phi_{T,j}}{T}^2
 \simeq \sum_{T\in\TT} h_T^2 \norm{\sum_{j\in\JJ} \phi_{T,j}}{T}^2
 = \sum_{T\in\TT} h_T^2 \norm{\phi}{T}^2 \lesssim \norm{\phi}{-1}^2.
\end{align*}
The proof is finished by combining the estimates. \qed

\subsection{Proof of upper bound in~\cref{eq:equiv:Hm1:Pp}}
Let $\phi_E\in \XX_E$, $\phi_{T,j}\in\XX_{T,j}$ be given.
Set $\phi := \phi_0 + \phi_1 := \sum_{E\in\EE} \phi_E + \sum_{T\in\TT}\sum_{j\in\JJ}\phi_{T,j}$.
Observe that $\phi_0 = \Pi^0\phi$ and $\phi_1 = (1-\Pi^0)\phi$. From~\cref{thm:Hm1} we already know that
\begin{align*}
  \norm{\phi}{-1}^2 \lesssim \norm{\phi_0}{-1}^2 + \norm{\phi_1}{-1}^2 
  \lesssim \sum_{E\in\EE} \norm{\phi_E}{-1}^2 
  + \norm{\phi_1}{-1}^2.
\end{align*}
By~\cref{lem:L2proj} we get
\begin{align*}
  \norm{\phi_1}{-1}^2 &= \norm{(1-\Pi^0)\phi}{-1}^2 
  \lesssim \norm{h(1-\Pi^0)\phi}{}^2 
  = \sum_{T\in\TT} h_T^2 \norm{\phi_1}{T}^2 = \sum_{T\in\TT} h_T^2 
  \norm{\sum_{j\in\JJ}\phi_{T,j}}{T}^2 
  \\ &\simeq \sum_{T\in\TT} h_T^2 
  \sum_{j\in\JJ} \norm{\phi_{T,j}}{T}^2 
  \lesssim \sum_{T\in\TT} \sum_{j\in\JJ} \norm{\phi_{T,j}}{-1}^2,
\end{align*}
which concludes the proof.
\qed

\section{Matrix representation}\label{sec:matrix}
In this section we briefly discuss the matrix representation of the preconditioner associated
to the space decompositions given in~\cref{sec:main}.

We consider the splitting $\PP^0(\TT) = \sum_{E\in\EE} \XX_E$. The additive Schwarz operator $\AS : \PP^0(\TT) \to
\PP^0(\TT)$ is given by 
$\AS = \sum_{E\in\EE} \AS_E$, where $\AS_E : \PP^0(\TT) \to \XX_E$ is the projection within
$H^{-1}(\Omega)$, i.e,
\begin{align*}
  \ip{\AS_E\phi}{\phi_E}_{-1} = \ip{\phi}{\phi_E}_{-1} \quad\text{for all } \phi_E\in \XX_E.
\end{align*}
Let $\ASmat_E$, $\ASmat$, $\GALmat_E$, $\GALmat$ denote the matrix representations of $\AS_E$, $\AS$,
$\ip\cdot\cdot_{-1}$ on $\XX_E$, and $\ip\cdot\cdot_{-1}$ on $\PP^0(\TT)$, respectively.
Moreover, let $\Imat_E$ denote the matrix form of the canonical embedding $\XX_E \to \PP^0(\TT)$.
Then, following standard references (e.g.,~\cite[Chapter~2]{ToselliWidlund}) simple calculations show that
\begin{align*}
  \GALmat_E \ASmat_E = \Imat_E^t \GALmat \quad\text{or equivalently}\quad
  \ASmat_E = \GALmat_E^{-1} \Imat_E^t \GALmat,
\end{align*}
and the overall matrix representation is
\begin{align*}
  \ASmat = \sum_{E\in\EE} \Imat_E \GALmat_E^{-1} \Imat_E^t \GALmat =: \overline\PRECmat^{-1} \GALmat.
\end{align*}
The theory on additive Schwarz operators, together with~\cref{thm:Hm1}, shows that the condition number of
$\ASmat$ is uniformly bounded, that is, $\overline\PRECmat$ and $\GALmat$ are spectrally equivalent.
Note that the dimension of the spaces $\XX_E$ is one so that $\GALmat_E$ is just a scalar. 
We can rewrite the preconditioner matrix $\overline\PRECmat$ as
\begin{align*}
  \overline\PRECmat^{-1} = \Imat \overline\Dmat \Imat^t,
\end{align*}
where the $j$-th column of $\Imat \in \R^{\#\TT \times \#\EE}$ is given by $\Imat_{E_j}$ and $\overline\Dmat\in \R^{\#\EE\times
\#\EE}$ is the diagonal matrix
with entries $\overline\Dmat_{jk} = \delta_{jk} \norm{\div\ppsi_{E_j}}{-1}^{-2}$.

Note that $\div\ppsi_E$ has support on at most two elements, i.e., with the notation from~\cref{sec:main}, 
\begin{align*}
  \div\ppsi_E =
  \begin{cases}
    \frac{|E|}{|T^+|} \chi_{T^+} - \frac{|E|}{|T^-|} \chi_{T^-} & \text{if } E\in \EE^\Omega
    \text{ where } E = \overline{T^+}\cap\overline{T^-}, \\
    \frac{|E|}{|T|} \chi_{T} &\text{if } E\in \EE^\Gamma \text{ where } E\subset \partial T.
  \end{cases}
\end{align*}
This means that each column of $\Imat$ has at most two non-zero entries and, thus, $\Imat$ is a sparse matrix.

Note that $\norm{\div\ppsi_E}{-1}$ is not computable in general. However, thanks to the additive Schwarz theory,
we can replace $\overline\Dmat$ by a matrix $\Dmat$ with equivalent entries. 
By~\cref{lem:scaling} it holds that
\begin{align*}
  \norm{\div\ppsi_E}{-1} \simeq \norm{\ppsi_E}{} \simeq |T|^{1/2} \simeq \diam(E)^{n/2} \simeq |E|^{n/(2(n-1))}.
\end{align*}
Defining $\Dmat_{jk} = \delta_{jk} |E_j|^{-n/(n-1)}$, this leads to the preconditioner
\begin{align*}
  \PRECmat^{-1} := \Imat \Dmat \Imat^t.
\end{align*}
It is spectrally equivalent to $\overline\PRECmat^{-1}$. Therefore, the condition number of
$\PRECmat^{-1}\GALmat$ is uniformly bounded.

The very same approach can be used to define a preconditioner associated to the splitting 
\begin{align*}
  \PP^0(\TT) = \XX_0 + \sum_{E\in\EE} \XX_E,
\end{align*}
considered in~\cref{sec:decomp:tildeHm1}.
Given the diagonal matrix $\widetilde\Dmat\in \R^{\#\EE^\Omega\times \#\EE^\Omega}$ 
with entries $\widetilde\Dmat_{jk}= \delta_{jk}|E_j|^{-n/(n-1)}$ (here, $E_j\in\EE^\Omega$), 
and the constant vector $\onevec\in \R^{\#\TT\times 1}$ (with $\onevec_j = 1$), define
\begin{align*}
  \widetilde\PRECmat^{-1} = \alpha \onevec \onevec^t + \widetilde\Imat \widetilde\Dmat \widetilde\Imat^t,
\end{align*}
where the columns of $\widetilde \Imat \in \R^{\#\TT\times\#\EE^\Omega}$ are given by $\Imat_E$ ($E\in\EE^\Omega$).
The constant $\alpha\simeq 1$ can be freely chosen. Then, $\widetilde\PRECmat$ is an optimal preconditioner for the
Galerkin matrix of $\ip{\cdot}\cdot_{-1,\sim}$ on $\PP^0(\TT)$.
Note that the matrix $\onevec \onevec^t$ is fully populated. However, in practical situations we are only interested in
the application of $\widetilde\PRECmat^{-1}$ to a vector $\mathbf{x}$, which can be implemented efficiently since
$\onevec\onevec^t$ has rank one.

Finally, for the higher order case from~\cref{sec:decomp:Hm1:Pp}, we define the preconditioner matrix
\begin{align*}
  \PRECmat_p^{-1} = \begin{pmatrix}
    \PRECmat^{-1} & \mathbf{0} \\
    \mathbf{0} & \Dmat^{(p)}
  \end{pmatrix},
\end{align*}
where, with $N_p := (d(n,p)-1)\#\TT$, $\Dmat^{(p)} \in \R^{N_p\times N_p}$ is a diagonal matrix.
Its entries are given by $(|T|^{1/n}\norm{\chi_{T,j}}{})^{-2}$ (with some appropriate ordering of the basis
functions).

The preconditioner matrix for the splitting considered in~\cref{sec:decomp:tildeHm1:Pp} reads
\begin{align*}
  \widetilde\PRECmat_p^{-1} = \begin{pmatrix}
    \widetilde\PRECmat^{-1} & \mathbf{0} \\
    \mathbf{0} & \Dmat^{(p)}
  \end{pmatrix},
\end{align*}
which can be obtained with the same argumentation.

\section{Experiments} \label{sec:num}
In this section we present some experiments with $p=0,1$ and $n=2,3,4$ on uniformly and locally refined meshes.
In order to show that our proposed preconditioners from~\cref{sec:matrix} lead to uniformly bounded condition numbers we
need a mechanism to produce the $H^{-1}(\Omega)$ and $\widetilde H^{-1}(\Omega)$ norms.
Here, we consider the discrete $H^{-1}$ inner product from the seminal work~\cite{BLP97} (which can be extended to
locally refined meshes, see~\cref{sec:discreteHm1norm}).

The condition numbers $\kappa(\cdot)$ which are displayed in the figures are 
obtained as follows. We use the power iteration (resp., inverse power iteration) to approximate 
the largest (resp., smallest) eigenvalue of a matrix
(which is symmetric with respect to some inner product)
and then report on the condition number of that matrix as the ratio of the approximated extreme eigenvalues.

\subsection{Discrete $H^{-1}(\Omega)$ and $\widetilde H^{-1}(\Omega)$ norms}\label{sec:discreteHm1norm}
Let $\widetilde Q : L^2(\Omega) \to \PP^1(\TT)\cap H_0^1(\Omega)$ denote the $L^2(\Omega)$
projection. Given that $\widetilde Q$ is bounded, and
\[
   \norm{\nabla \widetilde Q v}{} \lesssim \norm{\nabla v}{}
   \quad\text{and}\quad
   \norm{(1-\widetilde Q)v}{T} \lesssim \norm{h_T\nabla v}{\patch(T)}
   \quad \forall v\in H_0^1(\Omega)
\]
(see, e.g., \cite{KPP13} for locally refined meshes under consideration) we follow~\cite{BLP97}.
First, observe that for $\phi\in \PP^p(\TT)$
\begin{align*}
  \norm{(1-\widetilde Q)\phi}{-1} = \sup_{0\neq v\in H_0^1(\Omega)} 
  \frac{\ip{(1-\widetilde Q)\phi}{v}}{\norm{\nabla v}{}}
  = \sup_{0\neq v\in H_0^1(\Omega)} 
  \frac{\ip{\phi}{(1-\widetilde Q)v}}{\norm{\nabla v}{}} \lesssim \norm{h\phi}{}.
\end{align*}
Second, define $u[\phi] \in \PP^1\cap H_0^1(\Omega)$ by
\begin{align*}
  \ip{\nabla u[\phi]}{\nabla v} = \ip{\phi}{v} \quad\text{for all }v\in \PP^1(\TT)\cap H_0^1(\Omega).
\end{align*}
Third, using boundedness we get
\begin{align*}
  \norm{\widetilde Q\phi}{-1} &= \sup_{0\neq v\in H_0^1(\Omega)} \frac{\ip{\widetilde Q\phi}{v}}{\norm{\nabla v}{}}
  = \sup_{0\neq v\in H_0^1(\Omega)} \frac{\ip{\phi}{\widetilde Qv}}{\norm{\nabla v}{}}
  = \sup_{0\neq v\in H_0^1(\Omega)} \frac{\ip{\nabla u[\phi]}{\nabla \widetilde Qv}}{\norm{\nabla v}{}} 
  \lesssim \norm{\nabla u[\phi]}{}.
\end{align*}
Then, 
\begin{align*}
  \norm{\phi}{-1}^2 \lesssim \norm{\widetilde Q \phi}{-1}^2 + \norm{(1-\widetilde Q)\phi}{-1}^2 
  \lesssim \norm{\nabla u[\phi]}{}^2 + \norm{h\phi}{}^2 \lesssim \norm{\phi}{-1}^2,
\end{align*}
where in the last step we have used that $\norm{\nabla u[\phi]}{} \leq \norm{\phi}{-1}$ and the inverse
inequality~\cref{eq:invineq}.
Now let $\eta_j$ denote the nodal basis functions of $\PP^1(\TT)\cap H_0^1(\Omega)$ and let $\chi_k$ denote
the basis functions of $\PP^p(\TT)$, where $\chi_k = \chi_{T_k}$ for $k=1,\dots,\#\TT$.
We replace the mesh-width function $h$ by the equivalent function $\widetilde h$ given elementwise by $\widetilde h|_T
:= |T|^{1/n}$.
Defining the matrices
\begin{align*}
  \Mmat_{j\ell} := \ip{\chi_k}{\eta_j}, \quad \Rmat_{jk} := \ip{\nabla \eta_j}{\nabla \eta_k}, 
  \quad \Lmat_{\ell m} := \ip{\widetilde h^2\chi_\ell}{\chi_m}
\end{align*}
($j,k = 1,\dots,\dim(\PP^1(\TT)\cap H_0^1(\Omega))$, $\ell,m=1,\dots,d(n,p)\#\TT$) and
relating $\phi\in \PP^p(\TT)$ with $\mathbf{x}\in \R^{d(n,p)\#\TT\times 1}$ by
$\phi = \sum_{j=1}^{d(n,p)\#\TT} \mathbf{x}_j \chi_j$,
our considerations above yield
\begin{align*}
  \mathbf{x}^t (\Mmat^t\Rmat^{-1}\Mmat + \Lmat)\mathbf{x} = 
  \norm{\nabla u[\phi]}{}^2 + \norm{\widetilde h\phi}{}^2 \simeq
  \norm{\nabla u[\phi]}{}^2 + \norm{h\phi}{}^2 \simeq
  \norm{\phi}{-1}^2.
\end{align*}
Therefore, we replace the matrix $\GALmat$ from~\cref{sec:matrix} by the computable matrix
\begin{align*}
  \GALmat := \Mmat^t\Rmat^{-1}\Mmat + \beta \Lmat,
\end{align*}
where $\beta\simeq 1$ can be chosen freely.

Following the same argumentation (with obvious modifications) we replace $\widetilde\GALmat$ from~\cref{sec:matrix}
with the matrix
\begin{align*}
  \widetilde\GALmat := \widetilde\Mmat^t\widetilde\Rmat^{-1}\widetilde\Mmat + \beta \Lmat.
\end{align*}
Here,
\begin{align*}
  \widetilde\Mmat_{j\ell} := \ip{\chi_\ell}{\eta_j}, \quad
  \widetilde\Rmat_{jk} := \ip{\nabla \eta_j}{\nabla \eta_k} + \ip{\eta_j}{\eta_k}
\end{align*}
($j,k=1,\dots,\dim(\PP^1(\TT)\cap H^1(\Omega))$, $\ell = 1,\dots,d(n,p)\#\TT$) where
$\eta_j$ now refers to the nodal basis functions of $\PP^1(\TT)\cap H^1(\Omega)$.

\subsection{Condition numbers for $n=2$}

\begin{figure}
  \begin{center}
    \input{Cond2D}
  \end{center}
  \caption{Condition numbers for $n=2$ and $p=0$ (upper panel) resp. $p=1$ (lower panel).
  The left column corresponds to uniform refinements and the right one to local refinements. The results shown
  correspond to the decompositions from~\cref{sec:decomp:Hm1,sec:decomp:Hm1:Pp}.}
  \label{fig:2d}
\end{figure}

\begin{figure}
  \begin{center}
    \input{Cond2Dtilde}
  \end{center}
  \caption{Condition numbers for $n=2$ and $p=0$ (upper panel) resp. $p=1$ (lower panel).
  The left column corresponds to uniform refinements and the right one to local refinements. The results shown
  correspond to the decompositions from~\cref{sec:decomp:tildeHm1,sec:decomp:tildeHm1:Pp}.}
  \label{fig:2d:tilde}
\end{figure}

\begin{figure}
  \begin{center}
    \input{Cond3D}
  \end{center}
  \caption{Condition numbers for $n=3$ and uniform refinement. 
  The upper panel corresponds to $p=0$, the lower panel to $p=1$.}
  \label{fig:3d}
\end{figure}

\begin{figure}
  \begin{center}
    \input{Cond4D}
  \end{center}
  \caption{Condition numbers for $n=4$ and uniform refinement. 
  The upper panel corresponds to $p=0$, the lower panel to $p=1$.}
  \label{fig:4d}
\end{figure}

We consider the L-shaped domain $\Omega := (-1,1)^2\setminus (-1,0]^2$ with an initial triangulation of $12$ elements
of the same area.
Our refinement method is the newest vertex bisection (NVB). Uniform refinement means that we bisect each triangle twice
such that each father triangle is divided into four son elements.
To obtain some ``realistic'' locally refined meshes, we define $w(r,\varphi) := r^{2/3}\cos(2/3\varphi-\pi/6)$
with polar coordinates $(r,\varphi)$ centered at the origin.
This function has a singular behavior at the reentrant
corner of the domain $\Omega$ and corresponds to singularities of the Laplacian.
We compute the error indicators
\begin{align*}
  \mu(T) := \norm{(1-Q)w}{1,T}^2,
\end{align*}
where $Q: L^2(\Omega) \to \PP^1(\TT)\cap H^1(\Omega)$ denotes the $L^2(\Omega)$ projection.
A set (of minimal cardinality) $\MM\subseteq\TT$ is determined using the bulk criterion
\begin{align*}
  \frac14 \sum_{T\in\TT} \mu(T) \leq \sum_{T\in\MM} \mu(T).
\end{align*}
Then, $\TT$ is refined based on the set of marked elements $\MM$ using NVB. Further details on NVB can be found, e.g.,
in~\cite{KPP13,stevenson:NVB}.

The diagonal matrix $\mathbf{C}$ is defined as
\begin{align*}
  \mathbf{C}_{jk} := \delta_{jk} |T_j|^2.
\end{align*}
This choice is (up to some logarithmic factors) equivalent to  $\norm{\chi_j}{-1}^2 \simeq \norm{\chi_j}{-1,\sim}^2$ for
sufficiently small $|T_j|$, see~\cite[Theorem~4.8]{amt99} for the scaling of basis functions in negative order Sobolev
norms. 
We skipped the logarithmic factors since using them did not improve condition numbers.
For $p=1$ we use the matrix
\begin{align*}
  \mathbf{C}_1^{-1} := \begin{pmatrix}
    \mathbf{C}^{-1} & 0 \\
    0 & \Dmat^{(1)}
  \end{pmatrix}
\end{align*}
as diagonal preconditioner.

For the implementation of the matrices $\widetilde\PRECmat^{-1}$, $\GALmat$ and $\widetilde\GALmat$ we have set
the parameters to $(\alpha,\beta) = (1/100,1/10)$. We found that with this choice our proposed preconditioners lead to
reasonably small condition numbers for different examples (not reported here).
Nevertheless, one might find other values $(\alpha,\beta)$ that even lead to smaller condition numbers.

\Cref{fig:2d,fig:2d:tilde} show the condition numbers for uniform (left) and local refinements (right) for $p=0$ (upper panel) and $p=1$
(lower panel), respectively.
The diagonal preconditioner in the case of local refinements delivers --- as expected, see, e.g.,~\cite{amt99}, ---
condition numbers comparable to the case of uniform refinement. (More precisely, it is
shown in~\cite{amt99} that the condition number of diagonally preconditioned systems like those considered here only
depend on the number of elements up to some possible logarithmic factors.)
In all configurations our proposed preconditioners lead to quite small condition numbers, even on locally refined meshes, which confirms our
theoretical results.

\subsection{Condition numbers for $n=3$}

In this section we consider a similar problem in 3D where $\Omega = \{(-1,1)^2\setminus(-1,0]\}\times(0,1)$ is an L-shaped
domain. We start with a triangulation of 24 tetrahedrons.
The diagonal preconditioner matrix is now defined as
\begin{align*}
  \mathbf{C}_{jk} := \delta_{jk} |T_j|^{5/3}.
\end{align*}
The meshes are refined using the red refinement rule. 
(Now a uniform refinement corresponds to the division of one tetrahedra into eight tetrahedrons.)

From~\cref{fig:3d} we observe a similar behavior as in the case $n=2$. 
In particular, we see that our proposed preconditioners lead to quite small condition numbers (the parameters
$\alpha,\beta$ are chosen as in the case $n=2$).

\subsection{Condition numbers for $n=4$}
We consider the unit 4-cube $\Omega = (0,1)^4$ which is divided into 24 simplices (Kuhn's triangulation,
see~\cite{Bey00}). The diagonal preconditioner matrix is now defined with entries
\begin{align*}
  \mathbf{C}_{jk} := \delta_{jk} |T_j|^{3/2}.
\end{align*}
We use Freudenthal's algorithm (see also~\cite{Bey00})
to obtain uniform refined regular meshes. (Each simplex is decomposed into 16 subsimplices.)
We choose $\alpha=\tfrac1{10}=\beta$.
Results are shown in~\cref{fig:4d} for $p=0$ and $p=1$. Again, they appear to confirm our prediction
of bounded condition numbers.

\appendix
\section{Proof of~\cref{lem:localreg}}\label{app:localreg}
We follow exactly the same ideas and lines of proof as in~\cite[Appendix~A]{AinsworthGuzmanSayas} adapted to our
situation (with volume force but homogeneous boundary conditions) and notation. 

Throughout fix $T\in\TT$ and let $\cutoff\in C^\infty(\patch(T))$ denote a cut-off function with the properties
\begin{subequations}
\begin{align}
  \cutoff|_T &= 1,  \qquad
  \cutoff|_{\Omega\setminus\patch(T)} = 0, \\
  \norm{D^m \cutoff}{L^\infty(\patch(T))} & \lesssim h_T^{-m}, \quad\text{for } m=0,1,2.
\end{align}
\end{subequations}
Let $u\in H_0^1(\Omega)$ denote the solution of~\cref{eq:dirichlet} with datum $f\in L^2(\Omega)$.
Let $s = s(\Omega)\in(1/2,1]$ denote the regularity shift. Then, by~\cref{eq:regularity} we have that
\begin{align}\label{eq:app:reg}
  \norm{u}{1+s,T} \leq \norm{\cutoff u}{1+s,\patch(T)} = \norm{\cutoff u}{1+s} 
  \lesssim \norm{\Delta(\cutoff u)}{-1+s},
\end{align}
since $\cutoff u|_\Gamma = 0$ and $\Delta(\cutoff u)\in L^2(\Omega)$.

We consider the case where $|\partial \patch(T)\cap \Gamma| = 0$. 
Then, $\nabla(\cutoff u)\cdot\normal = 0$ on $\Gamma$. Let $v\in H^{1-s}(\Omega)$.
Using $v_{\patch(T)} := |\patch(T)|^{-1} \int_{\patch(T)} v \,dx$ and the product rule
\begin{align}\label{eq:app:prodrule}
  \Delta(\cutoff u) = u\Delta \cutoff + 2\nabla \cutoff\cdot\nabla u -\cutoff f,
\end{align}
we infer that
\begin{align*}
  \ip{\Delta(\cutoff u)}{v} &= \ip{\Delta(\cutoff u)}{v-v_{\patch(T)}} 
  = \ip{\Delta(\cutoff u)}{v-v_{\patch(T)}}_{\patch(T)} \\
  &= \ip{u\Delta \cutoff + 2\nabla\cutoff\cdot\nabla u - \cutoff f}{v-v_{\patch(T)}}_{\patch(T)}.
\end{align*}
Note that $\norm{v-v_{\patch(T)}}{\patch(T)} \lesssim h_T^{1-s}\snorm{v}{1-s,\patch(T)}$ and therefore,
\begin{align*}
  |\ip{\Delta(\cutoff u)}{v}| &\lesssim h_T^{1-s} (\norm{\Delta\cutoff}{L^\infty(\patch(T))}\norm{u}{\patch(T)} 
    + \norm{\nabla \cutoff}{L^\infty(\patch(T))}\norm{\nabla u}{\patch(T)} 
  + \norm{\cutoff}{L^\infty(\patch(T))} \norm{f}{\patch(T)}) \norm{v}{1-s} \\
  &\lesssim (h_T^{-1-s}\norm{u}{\patch(T)} + h_T^{-s}\norm{\nabla u}{\patch(T)} + h_T^{1-s}\norm{f}{\patch(T)})\norm{v}{1-s}.
\end{align*}
Recall that we consider the case where $\partial \patch(T)$ does not share a boundary facet.
So the same estimates hold true when we replace $u$ by $u=u-u_{\patch(T)}$ since $\cutoff (u-u_{\patch(T)})|_{\Gamma} =
0$ and $\Delta(\cutoff(u-u_{\patch(T)}))\in L^2(\Omega)$.
Using $\norm{u-u_{\patch(T)}}{\patch(T)}\lesssim h_T \norm{\nabla u}{\patch(T)}$, dividing by $\norm{v}{1-s}$,
and taking the supremum we get
\begin{align*}
  h_T^s\snorm{\nabla u}{s,T} \leq h_T^s\norm{u-u_{\patch(T)}}{1+s,T} \lesssim 
  h_T^s\norm{\Delta(\cutoff (u-u_{\patch(T)}))}{-1+s}
  \lesssim \norm{\nabla u}{\patch(T)} + h_T\norm{f}{\patch(T)}.
\end{align*}
Now we tackle the case where $\partial\patch(T)$ includes at least one boundary facet $E\in\EE^\Gamma$.
First, note that if a function $w$ vanishes on one facet $E$, then
\begin{align*}
  \norm{w}{\patch(T)} \lesssim h_T^r \snorm{w}{r,\patch(T)}, \quad 0\leq r\leq 1.
\end{align*}
Second, recall that
\begin{align*}
  \norm{\phi}{-1+s} = \sup_{0\neq v\in C_0^\infty(\Omega)} \frac{\ip{\phi}v}{\norm{v}{1-s}}.
\end{align*}
Then, the product rule~\cref{eq:app:prodrule} and the properties of the cut-off function prove that
\begin{align*}
  |\ip{\Delta(\cutoff u)}v| &\lesssim (h_T^{-2}\norm{u}{\patch(T)} + h_T^{-1}\norm{\nabla u}{\patch(T)} +
  \norm{f}{\patch(T)})\norm{v}{\patch(T)} \quad \text{for all } v \in C_0^\infty(\Omega).
\end{align*}
Using $\norm{v}{\patch(T)}\lesssim h_T^{1-s}\snorm{v}{1-s,\patch(T)}$, $\norm{u}{\patch(T)}\lesssim
h_T\norm{\nabla u}{\patch(T)}$ we further infer that
\begin{align*}
  |\ip{\Delta(\cutoff u)}v| &\lesssim h_T^{-s} (\norm{\nabla u}{\patch(T)} + h_T\norm{f}{\patch(T)})\norm{v}{1-s}.
\end{align*}
Dividing by $\norm{v}{1-s}$ and taking the supremum, we conclude that
\begin{align*}
  h_T^s\snorm{\nabla u}{s,T} \leq h_T^s\norm{u}{1+s,T} \lesssim h_T^s \norm{\Delta(\cutoff u)}{-1+s} 
  \lesssim \norm{\nabla u}{\patch(T)} + h_T \norm{f}{\patch(T)}.
\end{align*}
This finishes the proof of~\cref{lem:localreg}. 
\qed

\bibliographystyle{abbrv}
\bibliography{literature}

\end{document}

%% file: header.tex
\usepackage[capitalize,nameinlink]{cleveref}[0.19]
\crefname{section}{Section}{Sections}
\crefname{subsection}{Subsection}{Subsections}
\Crefname{section}{Section}{Sections}
\Crefname{subsection}{Subsection}{Subsections}
\Crefname{figure}{Figure}{Figures}
\crefformat{equation}{\textup{#2(#1)#3}}
\crefrangeformat{equation}{\textup{#3(#1)#4--#5(#2)#6}}
\crefmultiformat{equation}{\textup{#2(#1)#3}}{ and \textup{#2(#1)#3}}
{, \textup{#2(#1)#3}}{, and \textup{#2(#1)#3}}
\crefrangemultiformat{equation}{\textup{#3(#1)#4--#5(#2)#6}}%
{ and \textup{#3(#1)#4--#5(#2)#6}}{, \textup{#3(#1)#4--#5(#2)#6}}{, and \textup{#3(#1)#4--#5(#2)#6}}
\Crefformat{equation}{#2Equation~\textup{(#1)}#3}
\Crefrangeformat{equation}{Equations~\textup{#3(#1)#4--#5(#2)#6}}
\Crefmultiformat{equation}{Equations~\textup{#2(#1)#3}}{ and \textup{#2(#1)#3}}
{, \textup{#2(#1)#3}}{, and \textup{#2(#1)#3}}
\Crefrangemultiformat{equation}{Equations~\textup{#3(#1)#4--#5(#2)#6}}%
{ and \textup{#3(#1)#4--#5(#2)#6}}{, \textup{#3(#1)#4--#5(#2)#6}}{, and \textup{#3(#1)#4--#5(#2)#6}}
\crefdefaultlabelformat{#2\textup{#1}#3}

\newcommand{\cutoff}{\ensuremath{\eta}}

\newcommand{\AS}{\ensuremath{\mathcal{S}}}

\newcommand{\ASmat}{\ensuremath{\mathbf{S}}}
\newcommand{\PRECmat}{\ensuremath{\mathbf{B}}}
\newcommand{\GALmat}{\ensuremath{\mathbf{A}}}
\newcommand{\Imat}{\ensuremath{\mathbf{I}}}

\newcommand{\onevec}{\ensuremath{\mathbf{1}}}

\newcommand{\patch}{\ensuremath{\omega}}
\newcommand{\supp}{\ensuremath{\operatorname{supp}}}

\newcommand{\II}{\ensuremath{\mathcal{I}}}
\newcommand{\JJ}{\ensuremath{\mathcal{J}}}
\newcommand{\HHH}{\ensuremath{\mathcal{H}}}

\newtheorem{theorem}{Theorem}
\newtheorem{lemma}[theorem]{Lemma}

\newtheorem{proposition}[theorem]{Proposition}

\newcommand{\XX}{\mathcal{X}}
\newcommand{\YY}{\mathcal{Y}}
\DeclareMathOperator{\linhull}{span}

\def\enorm#1{|\hspace*{-.5mm}|\hspace*{-.5mm}|#1|\hspace*{-.5mm}|\hspace*{-.5mm}|}

\newcommand{\ip}[2]{(#1\hspace*{.5mm},#2)}

\newcommand{\norm}[3][]{#1\|#2#1\|_{#3}}

\newcommand{\snorm}[2]{|#1|_{#2}}

\newcommand{\diam}{\mathrm{diam}}
\def\div{{\rm div\,}}

\newcommand{\set}[2]{\big\{#1\,:\,#2\big\}}

\newcommand{\pp}{\boldsymbol{p}}

\newcommand{\RT}{\ensuremath{\mathcal{RT}}}

\newcommand{\R}{\ensuremath{\mathbb{R}}}
\newcommand{\N}{\ensuremath{\mathbb{N}}}

\newcommand{\TT}{\ensuremath{\mathcal{T}}}
\newcommand{\MM}{\ensuremath{\mathcal{M}}}

\newcommand{\PP}{\ensuremath{\mathcal{P}}}

\newcommand{\EE}{\ensuremath{\mathcal{E}}}

\newcommand{\zz}{\ensuremath{{\boldsymbol{z}}}}
\newcommand{\normal}{\ensuremath{{\boldsymbol{n}}}}

\newcommand{\Amat}{\ensuremath{\boldsymbol{A}}}

\newcommand{\ppsi}{{\boldsymbol\psi}}

\newcommand{\ssigma}{{\boldsymbol\sigma}}

\newcommand{\Dmat}{\ensuremath{\mathbf{D}}}

\newcommand{\Mmat}{\mathbf{M}}
\newcommand{\Lmat}{\mathbf{L}}
\newcommand{\Rmat}{\mathbf{R}}

\newcommand{\xx}{{\boldsymbol{x}}}

%% file: Cond2D.tex
\begin{tikzpicture}
\begin{loglogaxis}[
width=0.49\textwidth,
cycle list/Dark2-6,
cycle multiindex* list={
mark list*\nextlist
Dark2-6\nextlist},
every axis plot/.append style={ultra thick},
xlabel={number of elements $\#\TT$},
grid=major,
legend entries={\small $\kappa(\mathbf{C}^{-1}\GALmat)$,\small $\kappa(\PRECmat^{-1}\GALmat)$},
legend pos=north west,
]
\addplot table [x=nE,y=condDiag] {Hm1Uniform.dat};
\addplot table [x=nE,y=condP] {Hm1Uniform.dat};
\end{loglogaxis}
\end{tikzpicture}
\begin{tikzpicture}
\begin{loglogaxis}[
width=0.49\textwidth,
cycle list/Dark2-6,
cycle multiindex* list={
mark list*\nextlist
Dark2-6\nextlist},
every axis plot/.append style={ultra thick},
xlabel={number of elements $\#\TT$},
grid=major,
legend entries={\small $\kappa(\mathbf{C}^{-1}\GALmat)$,\small $\kappa(\PRECmat^{-1}\GALmat)$},
legend pos=north west,
]
\addplot table [x=nE,y=condDiag] {Hm1Adaptive.dat};
\addplot table [x=nE,y=condP] {Hm1Adaptive.dat};
\end{loglogaxis}
\end{tikzpicture}
\begin{tikzpicture}
\begin{loglogaxis}[
width=0.49\textwidth,
cycle list/Dark2-6,
cycle multiindex* list={
mark list*\nextlist
Dark2-6\nextlist},
every axis plot/.append style={ultra thick},
xlabel={number of elements $\#\TT$},
grid=major,
legend entries={\small $\kappa(\mathbf{C}_1^{-1}\GALmat)$,\small $\kappa(\PRECmat_1^{-1}\GALmat)$},
legend pos=north west,
]
\addplot table [x=nE,y=condDiag] {Hm1UniformP1.dat};
\addplot table [x=nE,y=condP] {Hm1UniformP1.dat};
\end{loglogaxis}
\end{tikzpicture}
\begin{tikzpicture}
\begin{loglogaxis}[
width=0.49\textwidth,
cycle list/Dark2-6,
cycle multiindex* list={
mark list*\nextlist
Dark2-6\nextlist},
every axis plot/.append style={ultra thick},
xlabel={number of elements $\#\TT$},
grid=major,
legend entries={\small $\kappa(\mathbf{C}_1^{-1}\GALmat)$,\small $\kappa(\PRECmat_1^{-1}\GALmat)$},
legend pos=north west,
]
\addplot table [x=nE,y=condDiag] {Hm1AdaptiveP1.dat};
\addplot table [x=nE,y=condP] {Hm1AdaptiveP1.dat};
\end{loglogaxis}
\end{tikzpicture}

%% file: Cond2Dtilde.tex
\begin{tikzpicture}
\begin{loglogaxis}[
width=0.49\textwidth,
cycle list/Dark2-6,
cycle multiindex* list={
mark list*\nextlist
Dark2-6\nextlist},
every axis plot/.append style={ultra thick},
xlabel={number of elements $\#\TT$},
grid=major,
legend entries={\small $\kappa(\mathbf{C}^{-1}\widetilde\GALmat)$,\small $\kappa(\widetilde\PRECmat^{-1}\widetilde\GALmat)$},
legend pos=north west,
]
\addplot table [x=nE,y=condDiag] {TildeHm1Uniform.dat};
\addplot table [x=nE,y=condP] {TildeHm1Uniform.dat};
\end{loglogaxis}
\end{tikzpicture}
\begin{tikzpicture}
\begin{loglogaxis}[
width=0.49\textwidth,
cycle list/Dark2-6,
cycle multiindex* list={
mark list*\nextlist
Dark2-6\nextlist},
every axis plot/.append style={ultra thick},
xlabel={number of elements $\#\TT$},
grid=major,
legend entries={\small $\kappa(\mathbf{C}^{-1}\widetilde\GALmat)$,\small $\kappa(\widetilde\PRECmat^{-1}\widetilde\GALmat)$},
legend pos=north west,
]
\addplot table [x=nE,y=condDiag] {TildeHm1Adaptive.dat};
\addplot table [x=nE,y=condP] {TildeHm1Adaptive.dat};
\end{loglogaxis}
\end{tikzpicture}
\begin{tikzpicture}
\begin{loglogaxis}[
width=0.49\textwidth,
cycle list/Dark2-6,
cycle multiindex* list={
mark list*\nextlist
Dark2-6\nextlist},
every axis plot/.append style={ultra thick},
xlabel={number of elements $\#\TT$},
grid=major,
legend entries={\small $\kappa(\mathbf{C}_1^{-1}\widetilde\GALmat)$,\small $\kappa(\widetilde\PRECmat_1^{-1}\widetilde\GALmat)$},
legend pos=north west,
]
\addplot table [x=nE,y=condDiag] {TildeHm1UniformP1.dat};
\addplot table [x=nE,y=condP] {TildeHm1UniformP1.dat};
\end{loglogaxis}
\end{tikzpicture}
\begin{tikzpicture}
\begin{loglogaxis}[
width=0.49\textwidth,
cycle list/Dark2-6,
cycle multiindex* list={
mark list*\nextlist
Dark2-6\nextlist},
every axis plot/.append style={ultra thick},
xlabel={number of elements $\#\TT$},
grid=major,
legend entries={\small $\kappa(\mathbf{C}_1^{-1}\widetilde\GALmat)$,\small $\kappa(\widetilde\PRECmat_1^{-1}\widetilde\GALmat)$},
legend pos=north west,
]
\addplot table [x=nE,y=condDiag] {TildeHm1AdaptiveP1.dat};
\addplot table [x=nE,y=condP] {TildeHm1AdaptiveP1.dat};
\end{loglogaxis}
\end{tikzpicture}

%% file: Cond3D.tex
\begin{tikzpicture}
\begin{loglogaxis}[
width=0.49\textwidth,
cycle list/Dark2-6,
cycle multiindex* list={
mark list*\nextlist
Dark2-6\nextlist},
every axis plot/.append style={ultra thick},
xlabel={number of elements $\#\TT$},
grid=major,
legend entries={\small $\kappa(\mathbf{C}^{-1}\GALmat)$,\small $\kappa(\PRECmat^{-1}\GALmat)$},
legend pos=north west,
]
\addplot table [x=nE,y=condDiag] {Hm1Uniform3D.dat};
\addplot table [x=nE,y=condP] {Hm1Uniform3D.dat};
\end{loglogaxis}
\end{tikzpicture}
\begin{tikzpicture}
\begin{loglogaxis}[
width=0.49\textwidth,
cycle list/Dark2-6,
cycle multiindex* list={
mark list*\nextlist
Dark2-6\nextlist},
every axis plot/.append style={ultra thick},
xlabel={number of elements $\#\TT$},
grid=major,
legend entries={\small $\kappa(\mathbf{C}^{-1}\widetilde\GALmat)$,\small $\kappa(\widetilde\PRECmat^{-1}\widetilde\GALmat)$},
legend pos=north west,
]
\addplot table [x=nE,y=condDiag] {TildeHm1Uniform3D.dat};
\addplot table [x=nE,y=condP] {TildeHm1Uniform3D.dat};
\end{loglogaxis}
\end{tikzpicture}
\begin{tikzpicture}
\begin{loglogaxis}[
width=0.49\textwidth,
cycle list/Dark2-6,
cycle multiindex* list={
mark list*\nextlist
Dark2-6\nextlist},
every axis plot/.append style={ultra thick},
xlabel={number of elements $\#\TT$},
grid=major,
legend entries={\small $\kappa(\mathbf{C}_1^{-1}\GALmat)$,\small $\kappa(\PRECmat_1^{-1}\GALmat)$},
legend pos=north west,
]
\addplot table [x=nE,y=condDiag] {Hm1UniformP13D.dat};
\addplot table [x=nE,y=condP] {Hm1UniformP13D.dat};
\end{loglogaxis}
\end{tikzpicture}
\begin{tikzpicture}
\begin{loglogaxis}[
width=0.49\textwidth,
cycle list/Dark2-6,
cycle multiindex* list={
mark list*\nextlist
Dark2-6\nextlist},
every axis plot/.append style={ultra thick},
xlabel={number of elements $\#\TT$},
grid=major,
legend entries={\small $\kappa(\mathbf{C}_1^{-1}\widetilde\GALmat)$,\small $\kappa(\widetilde\PRECmat_1^{-1}\widetilde\GALmat)$},
legend pos=north west,
]
\addplot table [x=nE,y=condDiag] {TildeHm1UniformP13D.dat};
\addplot table [x=nE,y=condP] {TildeHm1UniformP13D.dat};
\end{loglogaxis}
\end{tikzpicture}

%% file: Cond4D.tex
\begin{tikzpicture}
\begin{loglogaxis}[
width=0.49\textwidth,
cycle list/Dark2-6,
cycle multiindex* list={
mark list*\nextlist
Dark2-6\nextlist},
every axis plot/.append style={ultra thick},
xlabel={number of elements $\#\TT$},
grid=major,
legend entries={\small $\kappa(\mathbf{C}^{-1}\GALmat)$,\small $\kappa(\PRECmat^{-1}\GALmat)$},
legend pos=north west,
]
\addplot table [x=nE,y=condDiag] {Hm1Uniform4D.dat};
\addplot table [x=nE,y=condP] {Hm1Uniform4D.dat};
\end{loglogaxis}
\end{tikzpicture}
\begin{tikzpicture}
\begin{loglogaxis}[
width=0.49\textwidth,
cycle list/Dark2-6,
cycle multiindex* list={
mark list*\nextlist
Dark2-6\nextlist},
every axis plot/.append style={ultra thick},
xlabel={number of elements $\#\TT$},
grid=major,
legend entries={\small $\kappa(\mathbf{C}^{-1}\widetilde\GALmat)$,\small $\kappa(\widetilde\PRECmat^{-1}\widetilde\GALmat)$},
legend pos=north west,
]
\addplot table [x=nE,y=condDiag] {TildeHm1Uniform4D.dat};
\addplot table [x=nE,y=condP] {TildeHm1Uniform4D.dat};
\end{loglogaxis}
\end{tikzpicture}
\begin{tikzpicture}
\begin{loglogaxis}[
width=0.49\textwidth,
cycle list/Dark2-6,
cycle multiindex* list={
mark list*\nextlist
Dark2-6\nextlist},
every axis plot/.append style={ultra thick},
xlabel={number of elements $\#\TT$},
grid=major,
legend entries={\small $\kappa(\mathbf{C}_1^{-1}\GALmat)$,\small $\kappa(\PRECmat_1^{-1}\GALmat)$},
legend pos=north west,
]
\addplot table [x=nE,y=condDiag] {Hm1UniformP14D.dat};
\addplot table [x=nE,y=condP] {Hm1UniformP14D.dat};
\end{loglogaxis}
\end{tikzpicture}
\begin{tikzpicture}
\begin{loglogaxis}[
width=0.49\textwidth,
cycle list/Dark2-6,
cycle multiindex* list={
mark list*\nextlist
Dark2-6\nextlist},
every axis plot/.append style={ultra thick},
xlabel={number of elements $\#\TT$},
grid=major,
legend entries={\small $\kappa(\mathbf{C}_1^{-1}\widetilde\GALmat)$,\small $\kappa(\widetilde\PRECmat_1^{-1}\widetilde\GALmat)$},
legend pos=north west,
]
\addplot table [x=nE,y=condDiag] {TildeHm1UniformP14D.dat};
\addplot table [x=nE,y=condP] {TildeHm1UniformP14D.dat};
\end{loglogaxis}
\end{tikzpicture}